\documentclass{amsart}
\usepackage[left=3cm, right=3cm]{geometry}
\usepackage{amsmath,amsfonts,amssymb,amsthm}
\usepackage{graphicx}
\usepackage[english]{babel}
\usepackage[utf8]{inputenc}
\usepackage{hyperref}
\usepackage{csquotes}
\usepackage{xcolor}
\usepackage{pgfplots, tikz}
\usepackage{cleveref}
\usepackage{booktabs}
\usepackage{refcount}

\newtheorem{theorem}{Theorem}
\newtheorem{lemma}{Lemma}[section]
\newtheorem{corollary}[lemma]{Corollary}
\newtheorem{proposition}[lemma]{Proposition}
\newtheorem*{observation}{Observation}

\theoremstyle{definition}
\newtheorem{definition}[lemma]{Definition}

\theoremstyle{remark}
\newtheorem{remark}[lemma]{Remark}

\newcommand{\C}{\mathbb{C}}
\newcommand{\Z}{\mathbb{Z}}
\newcommand{\N}{\mathbb{N}}
\newcommand{\Q}{\mathbb{Q}}
\newcommand{\R}{\mathbb{R}}
\newcommand{\M}{\mathcal{M}}
\newcommand{\D}{\mathbb{D}}

\newcommand{\Cinf}{\overline{\mathbb{C}}}
\newcommand{\DoDef}{\mathcal{D}}
\newcommand{\Htop}{\mathcal{H}}
\newcommand{\Ncal}{\mathcal{N}_0}
\newcommand{\WB}{\mathcal{WB}}
\newcommand{\ind}[2]{\imath(#1,#2)}
\newcommand{\jind}[2]{\jmath(#1,#2)}
\newcommand{\Imag}{\operatorname{Im}}
\newcommand{\Real}{\operatorname{Re}}
\newcommand{\gate}[2][]{\mathbf{gate}_{#1}(#2)}
\newcommand{\lph}[2]{\tilde{\tau}_{#1}(#2)}

\newcommand{\hide}[1]{}

\title{Pi in the Mandelbrot set everywhere}
\author{Thies Brockmoeller \and Oscar Scherz \and Nedim Srkalovic}
\date{\today}

\begin{document}

\begin{abstract}
    The numerical phenomenon of $\pi$ appearing at parameters $c = 1/4$, $c=-3/4$ and $c=-5/4$ in the Mandelbrot set $\M$ has been known for over 30 years. In 2001, the first proof was provided in \cite{klebanoff} for the parameter $c=1/4$. Very recently in 2023, an even sharper result for $c=1/4$ was proved using holomorphic dynamics in \cite{paul}. This new proof also provided a conceptual understanding of the phenomenon. In this paper, we give, for the first time, a proof of the known phenomenon for the parameters  $c=-3/4$ and $c=-5/4$, which is also conceptual, and we provide a generalization of the phenomenon and the proof for all bifurcation points of the Mandelbrot set.
\end{abstract}

\maketitle

\section{Introduction}

The fact that ``$\pi$ occurs naturally in the Mandelbrot set'' has raised quite some attention since the 1980s or 1990s, not long after the Mandelbrot set itself was discovered and had become immensely popular. Many people independently observed the following.

\begin{observation}
	Fix an ``escape radius'' $R\geq2$ and consider the parameter $c_\alpha:=-3/4+it$, for $t\in\R\setminus\{0\}$. Define the ``escape time'' $N(c_t)$ as the minimal index in the sequence
	\begin{equation}
			z_0:=0, \quad  z_{n+1}=z_n^2+c_t
	\label{Eq:MandelIter}
	\end{equation}
	for which $|z_n|>R$. Then $N(c_t)\cdot |t|\to\pi$. 
	
	Similarly, for the parameter $c'_t:=1/4+t$ we have $N(c'_t)\cdot |t|^{1/2}\to\pi$ and for $c''_t := -5/4 - t^2 + it$ we have $N(c''_t)\cdot|t|\to\frac{\pi}{2}$.
\end{observation}

More specifically, people have tabulated $N(c_t)$ and $N(c'_t)$ for the natural choice $t=10^{-n}$. Then the decimal digits of $N(c_t)$ are the first $n$ decimal digits of $\pi$ (and similarly the decimal digits of $N(c''_t)$ yield the first $n$ decimal digits of $\pi/2$), while the decimal digits of $N(c'_t)$, when $n$ even, are the first $n/2$ decimal digits of $\pi$. 

This occurrence of $\pi$ was considered mysterious by many, especially since the experiment itself is quite natural. The relevant parameters $c_t$ and $c'_t$ are indicated in Figure~\ref{Fig:MandelbrotPi}. In particular, people studied the question (still open, for all we know) whether the vertical line $c=-3/4+it$ intersects the Mandelbrot set in more points than just $c=-3/4$.

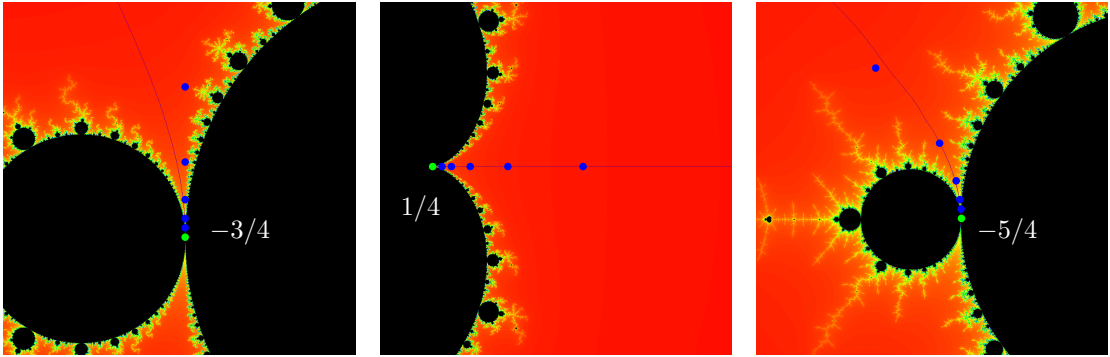
\begin{figure}[htbp]
\pgfdeclarelayer{background layer}
\pgfsetlayers{background layer,main}
\begin{tikzpicture}
    \fill[blue] (-1.64+2,0.16) circle (0.05);
    \fill[blue] (-1.64+1,0.16) circle (0.05);
    \fill[blue] (-1.64+0.5,0.16) circle (0.05);
    \fill[blue] (-1.64+0.25,0.16) circle (0.05);
    \fill[blue] (-1.64+0.125,0.16) circle (0.05);
    \fill[green] (-1.64,0.16) circle (0.05);
    \node[white] at (-1.8,-0.4) {$1/4$};
    \fill[blue] (-4.93,-0.78+2) circle (0.05);
    \fill[blue] (-4.93,-0.78+1) circle (0.05);
    \fill[blue] (-4.93,-0.78+0.5) circle (0.05);
    \fill[blue] (-4.93,-0.78+0.25) circle (0.05);
    \fill[blue] (-4.93,-0.78+0.125) circle (0.05);
    \fill[green] (-4.93,-0.78) circle (0.05);
    \node[white] at (-4.2,-0.7) {$-3/4$};
    \fill[blue] (4.25,-0.53+2) circle (0.05);
    \fill[blue] (5.1,-0.53+1) circle (0.05);
    \fill[blue] (5.32,-0.53+0.5) circle (0.05);
    \fill[blue] (5.37,-0.53+0.25) circle (0.05);
    \fill[blue] (5.385,-0.53+0.125) circle (0.05);
    \fill[green] (5.39,-0.53) circle (0.05);
    \node[white] at (6,-0.7) {$-5/4$};
    \begin{pgfonlayer}{background layer}
    \node at (-5,0) {\includegraphics[width=0.3\linewidth]{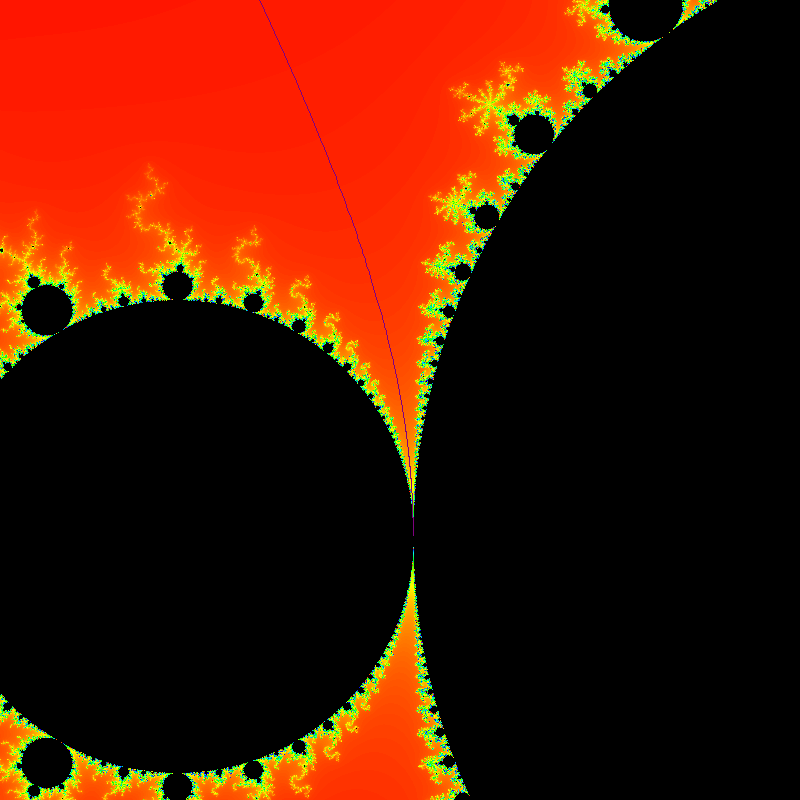}};
    \node at (0,0) {\includegraphics[width=0.3\linewidth]{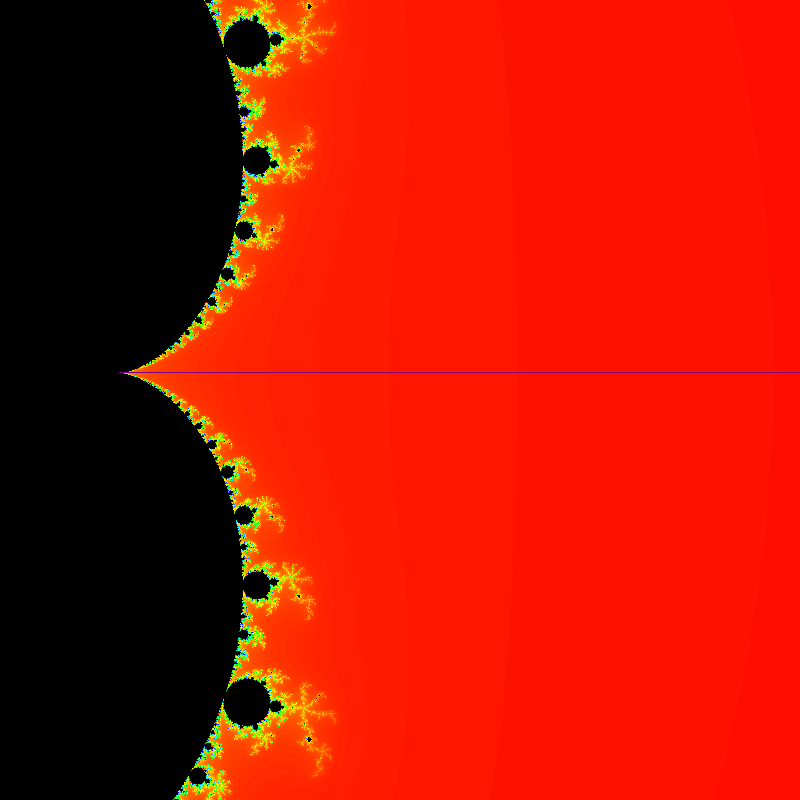}};
    \node at (5,0) {\includegraphics[width=0.3\linewidth]{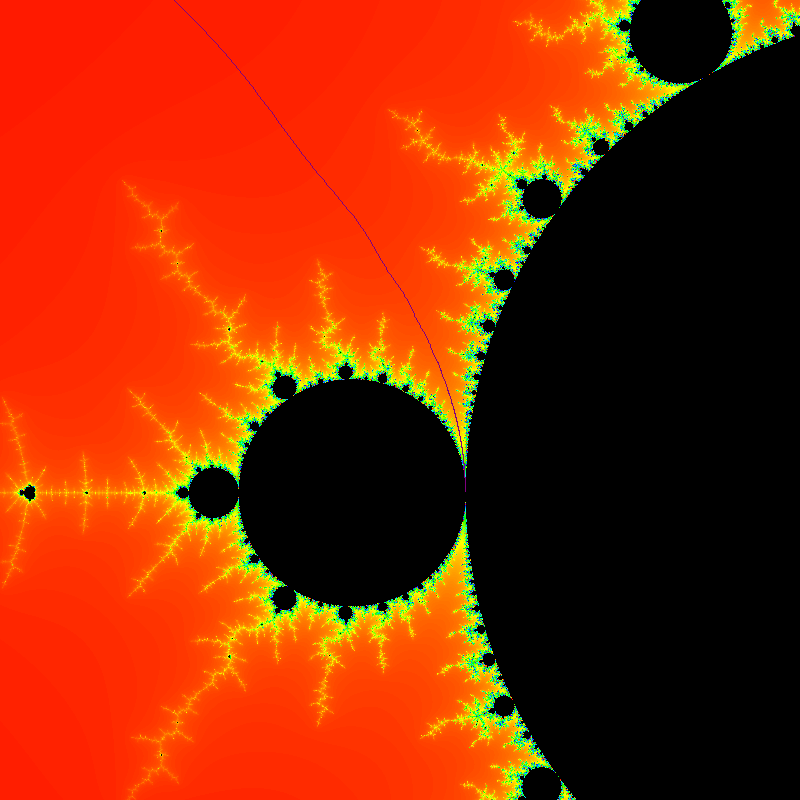}};
    \end{pgfonlayer}
\end{tikzpicture}
\caption{The parameters $c=-3/4$, $c'=1/4$ and $c''=-5/4$ in the Mandelbrot set, together with the parameter rays $\mathcal{R}_{1/3}$, $\mathcal{R}_{0/1}$ and $\mathcal{R}_{6/15}$ and sequences $c_t$, $c'_t$ and $c''_t$ approaching the bifurcations. The images are taken with \cite{mandel}.}	
\label{Fig:MandelbrotPi}
\end{figure}

Let us review the underlying equations. The famous Mandelbrot set is defined as
\begin{equation}
	\M :=\{c\in\C\mid \text{the sequence $z_0=0, z_{n+1}=z_n^2+c$ is bounded}\}. 
\end{equation}
It is well known that all $c\in\M$ satisfy $|c|\le 2$, and that $|z_n|>2$ for a single $n$ implies that the sequence diverges, so the underlying parameter $c$ is not in $\M$. Therefore, for any $c\in\C$ one iterates the sequence until one has $|z_n|>R$ for an arbitrary $R\ge 2$. If there is such an $n$, then $c\not\in\M$, and otherwise one stops the iteration after some predetermined iteration bound and declares $c$ as an ``approximate'' element of $\M$. 

The occurrence of $\pi$ is thus the outcome of a natural experiment to study $\M$ computationally. While the fact itself has been known for a long time, the only published proof appeared in 2001, more than a decade after the observation itself \cite{klebanoff}. It only covers the single case $c'=1/4$, it is not conceptual, and it does not provide an insightful explanation of why to expect $\pi$.

In 2015, a \emph{Numberphile} video was published that presents the observation, but it does not present either a proof or an explanation \cite{numberphile}. This video was widely watched and together with a second one attracted more than a million views, highlighting the impact of the phenomenon on a general audience, even without a conceptual explanation.

Very recently, in 2023, Siewert \cite{paul} developed a new proof and a conceptual explanation for the occurrence of $\pi$ at the single point $c'=1/4$, based on parabolic perturbation theory of \emph{simple} parabolic maps as developed by Shishikura \cite{shishikura_2000}. 

In this note, we provide the first proof for the original observation at the parameter $c=-3/4$. Moreover, we study the analogous question at all the infinitely many bifurcation parameters in the Mandelbrot set, we observe that an analogous result holds for all of them, and we develop a conceptual general proof of this general fact. Of course, the general result has to specify a natural sequence of points that converges to any given bifurcation parameter, and it has to take the ``natural scale'' at this parameter into account.

A \emph{parabolic parameter} is a point $c_0\in\M$ such that the polynomial $p_{c_0}(z)= z^2+c_0$ has two colliding periodic orbits (of equal or different periods). It is well known that this is equivalent to the fact that $p_{c_0}$ has a \emph{parabolic periodic point}, that is a point $z_0$ such that $p_{c_0}^{\circ n}(z_0)=z_0$ and $(p_{c_0}^{\circ n})'(z_0)$ is a $q$th root of unity (for details, see \Cref{sec:background}). It is well known that each parabolic parameter is of one of two possible types, \emph{primitive} if $q=1$ or \emph{satellite} if $q \geq2$. 

We can then state our main result as follows.

\begin{theorem}
\label{Thm:One}
Let $c_0$ be a parabolic parameter in $\M$ with parabolic periodic point $z_0$ and let $n$ be its period. Then $(p^{\circ n})'(z_0)$ is a $q$th root of unity. Fix an escape radius $R>2$. Let $c_k\in\C\setminus\M$ be a sequence such that $c_k\to c_0$ along one of the two parameter rays that land at $c_0$, and let $N(c_k)$ be the escape time of $c_k$. Then
\begin{itemize}
	\item if $c_0$ is of satellite type, then $N(c_k)\cdot |c_k-c_0|\cdot \tau(c_0)\longrightarrow\pi$;
	\item if $c_0$ is primitive, then $N(c_k)\cdot |c_k-c_0|^{1/2}\cdot \tau(c_0)\longrightarrow\pi$.
\end{itemize}
Here $\tau(c_0) = \frac{|\mu_{qn}'(c_0)|}{2qn}$ where $\mu_{qn}'(c_0)$ is the derivative of the natural multiplier map of the hyperbolic component whose root is $c_0$. Note that in the primitive case, we have $q=1$.
\end{theorem}

In fact, we can bound the speed of divergence more precisely, see Theorem~\ref{the:main theorem} for a more detailed statement.

The proof relies on parabolic bifurcation theory which describes a ``perturbed version'' of the well-known Flower Theorem. Hence, we still find attracting and repelling petals, however, now, points in attracting petals may not converge to the parabolic fixed point but instead pass through a ``gate'' formed by two fixed points. We then study curves that are orthogonal to the dynamics. These curves end at the two fixed points with a certain ``tangent angle''. Iterates of these curves then pass through the gate, while their ends rotate around the fixed point each iteration, the angle determined by the fixed point's multiplier $\mu$. In total, a rotation of $2\pi$ is done as the curves iterate from one attracting petal to one repelling petal. Thus, the escape time $N$ of the critical orbit starting in some attracting petal can be estimated by $N\approx 2\pi/\arg\mu$. The rigorous approach to this is provided in theories such as \cite{shishikura_2000} and \cite{oudkerk}. Also, see \Cref{rem:conceptual} and \Cref{fig:conceptual} for the conceptual understanding behind \Cref{Thm:One}.

Observe that the parameters $c=-3/4$ and $c''=-5/4$ are of satellite type, while $c'=1/4$ is primitive. $c$ and $c'$ have $\tau(c)=\tau(c')=1$ while $\tau(c'')=1/2$, so our result includes the one known result for $c=1/4$ (the associated parameter ray is $(1/4,\infty)\subset\R$), and it is analogous to the known results for $c=-3/4$ and $c''=-5/4$: the difference is that convergence to $ c$ and $c''$ in our approach is along the parameter rays (which is dynamically natural), while in the original experiments it was along a vertical line at $c=-3/4$ or parabola at $c''=-5/4$. Moreover, to our knowledge, the question of whether the vertical line $c_t = -3/4 + it$ intersects $\M$ only at $c=-3/4$ is still open.

This last issue raises the question of just how precisely one needs to choose the parameters that converge to a limiting parabolic parameter. The answer is ``not precisely at all''. This is our second result.

\begin{theorem}[Near the parameter ray]\label{Thm:Two}
    Let $c_0$ be a satellite bifurcation. Let $(c_n)\in\C\setminus\M$ be a sequence along one of the parameter rays of $c_0$ with $c_n\to c_0$. Then there exists a real number $1 < a \leq 8$ such that \Cref{Thm:One} remains true for a sequence $\tilde{c}_n\to c_0$ with the property that for $r_n := a|\tilde{c}_n - c_n|$ the disks $D_{r_n}(c_n)$ do not intersect $\M$.
\end{theorem}

We believe that \Cref{Thm:Two} is true for all $a>1$ which is a stronger statement. Geometrically for $a=2$, this means that the $\tilde c_n$ are closer to $c_n$ than to $\M$.
In this situation, the escape times for $\tilde c_n$ and $c_n$ differ at most by a constant that depends only on $a$.

In all statements, the precise choice of the escape radius $R$ is irrelevant: the escape times for different values of $R$ differ only by constants.

\subsection{Background}\label{sec:background}

Some background in complex dynamics and the Mandelbrot set is given here. In the following definitions let $f:U\to U$ be some holomorphic map on a domain $U\subseteq\C$. Denote the $n$th iteration of $f$ by $f^{\circ n}$. Also define $p_c(z) := z^2 + c$ for some $c\in\M$.

\begin{definition}[Orbit, periodic points, multiplier]
    Let $z_0\in U$. Then the \emph{orbit of $z_0$ under the map $f$} is the sequence $(z_n) = (f^{\circ n}(z_0))$.
    
    We call $z_0$ a \emph{periodic point of exact period $n$} if $n\geq1$ is the smallest number such that $f^{\circ n}(z_0) = z_0$. If $n=1$, then $z_0$ is a \emph{fixed point} of $f$. Let $\mu := (f^{\circ n})'(z_0)$ be the \emph{multiplier} of $z_0$. We call $z_0$
    \begin{itemize}
        \item \emph{superattracting} if $|\mu| = 0$,
        \item \emph{attracting} if $|\mu| < 1$,
        \item \emph{indifferent} if $|\mu| = 1$,
        \item and \emph{repelling} if $|\mu| > 1$.
    \end{itemize}
\end{definition}

\begin{definition}[Parabolic periodic point]
   Let $z_0\in U$ be a periodic point of exact period $n$ under the map $f$. We call $z_0$ a \emph{parabolic} periodic point if there exists an $r\in\Q/\Z$ such that $(f^{\circ n})'(z_0) = e^{2\pi i r}$.
\end{definition}

\begin{definition}[Hyperbolic components, multiplier map]
    A \emph{hyperbolic component of period $n$} of the Mandelbrot is a connected component of the set of parameters $c$ such that $p_c$ has an attracting periodic point of period $n$.
    
    Let $H_n$ be such a component. Then define the biholomorphic \emph{multiplier map} $\mu_n:H_n\to\D$ as the multiplier of the unique attracting periodic point of the parameter $c$. One can extend the multiplier map to the homeomorphism $\mu_n:\overline{H}_n\to\overline{\D}$. Then the \emph{root} of $H_n$ is $c_0 := \mu_n^{-1}(+1)$.
\end{definition}

\begin{definition}[Parabolic parameter]
    Let $c_0\in\M$. If $p_{c_0}$ has a parabolic periodic point, then we call $c_0$ a parabolic parameter.
\end{definition}

\begin{proposition}[Bifurcations]
    If $p_{c_0}$ has a parabolic periodic point $z_0$ of exact period $n$ and with multiplier $\mu = e^{2\pi i p/q} \neq 1$ with $p$ and $q$ coprime, then $c_0$ is on the boundary of a hyperbolic component of period $n$ and of another hyperbolic component of period $qn$.
    
    The parabolic periodic point $z_0$ breaks up under a perturbation $c$ close to $c_0$ into one orbit of exact period $n$ of $p_c$ and one orbit of exact period $qn$ of $p_c$.
\end{proposition}

\begin{definition}[Bifurcation]
        We call $c_0$ a \emph{bifurcation parameter of period $n$ to period $qn$} if $c_0$ is a common boundary point of hyperbolic components of period $n$ and $qn$.
\end{definition}

\begin{proposition}[Roots of hyperbolic components]
    If $p_{c_0}$ has a parabolic periodic point $z_0$ of exact period $n$, then it is on the boundary of a hyperbolic component of period $n$. There are two cases:
    \begin{itemize}
        \item \emph{primitive case:} the multiplier $\mu = (p_{c_0}^{\circ n})'(z_0) = +1$; then $c_0$ is the root of a hyperbolic component of period $n$;
        \item \emph{satellite case (bifurcation):} the multiplier $\mu = (p_{c_0}^{\circ n})'(z_0)$ is a $q$th root of unity with $q\geq2$; then $c_0$ is on the boundary of a hyperbolic component of period $n$, and the root of a hyperbolic component of period $qn$.
    \end{itemize}
\end{proposition}

\begin{definition}[Dynamic rays]
    Let $\varphi_c$ denote the Böttcher map of $p_c$. Then the \emph{dynamic ray at angle $\vartheta\in\R/\Z$} is the set $R_{p_c}(\vartheta) := \varphi_c^{-1}\left((1,\infty)\cdot e^{2\pi i\vartheta}\right)$.
\end{definition}

\begin{definition}[Riemann map of $\C\setminus\M$]
    Let $\Phi_\M:\C\setminus\M\to\C\setminus\overline{\D}$ be the Riemann map of the complement of $\M$ as defined by Douady and Hubbard in \cite{orsay}.
\end{definition}

\begin{definition}[Parameter rays]
    A \emph{parameter ray at angle $\vartheta\in\R/\Z$} is the set $\mathcal{R}_\vartheta := \Phi_\M^{-1}\left((1,\infty)\cdot e^{2\pi i\vartheta}\right)$.
\end{definition}
It is known that exactly two parameter rays land at each parabolic parameter $c_0\in\M$ (except for $c_0=1/4$).

\section{Overview of Oudkerk's theory}\label{sec:oudkerk}

The theory of parabolic perturbation was initially developed by Douady in \cite{orsay} to study the behavior of the Mandelbrot set close to parabolic parameters. It was further developed by Shishikura in \cite{shishikura_2000}. Both study maps of parabolic periodic points conjugate to $z\mapsto z + z^2 + O(z^3)$ and perturbations thereof. However, in the Mandelbrot set only parabolic parameters of the primitive type are described by these theories. 

Hence, when studying parabolic parameters of bifurcation type, a more general theory describing perturbations of maps $z\mapsto z + z^{q+1} + O\left(z^{q+2}\right)$ with integers $q\geq1$ is needed. This was provided by Oudkerk in his still unpublished \cite{oudkerk} who built upon the previously mentioned work. This is why our results are based on his theory and why the following section is dedicated to giving an overview of the necessary definitions and statements.

We will look at a class of maps called \emph{well behaved}. These maps exhibit many useful properties due to their relatively easily constructible Fatou coordinates. We will directly cite definitions and other explanations from \cite{oudkerk}.

\begin{definition}[Compact-open topology; see {\cite[Definition~2.1.1]{oudkerk}}] 
    For any holomorphic map $f$ defined on an open subset of $\Cinf$ let $\DoDef(f)$ denote the domain of definition of $f$. Now set
    \[
    \Htop:=\left\{ \text{$(f,\DoDef(f))\mid \DoDef(f)\subset\C$ open and  $f\colon \DoDef(f)\to\C$ holomorphic} \rule{0pt}{11pt} \right\}.
    \]
    Define a (non-Hausdorff) topology on $\Htop$ so that, given $\varepsilon>0$ and $K\subset\DoDef(f)$ compact, a neighborhood of $f_1\in\Htop$ is the set
    \[
        N(f_1, K, \varepsilon) := \{(g,\DoDef(g))\in\Htop\mid K\subset{\DoDef}(g),
        \|f_1-g\|_K < \varepsilon \}
        \;.
    \]
\end{definition}
    
    In this topology,    
    $f_m\to f$ if and only if for every compact set $K\subset\DoDef(f)$ there is an $m_0$ so that $K\subset\DoDef(f_m)$ for every $m\geq m_0$, and $f_m|_K\to f_K$ uniformly as $m\to +\infty$.     
    Roughly speaking, this means that $f_m$ converges to $f$ uniformly on compact subsets of $\DoDef(f)$.

    We will write $f_k\to f$ for some $(f,\DoDef(f))\in\Htop$ to mean convergence in this compact-open topology.

We take a small $r_0>0$ such that $K_0:=\overline{D_{2r_0}}\subseteq\DoDef(f_0)$. We also take a small open neighborhood $\Ncal$ of $f_0$ in the compact-open topology such that $K_0\subseteq\DoDef(f)$ for all $(f,\DoDef(f))\in\Ncal$.

The following definitions and propositions will rely on these fixed choices of $r_0$, $K_0$, and $\Ncal$. However, when applying Oudkerk's theory in \Cref{sec:application}, $r_0$, etc. can be chosen relatively freely since the maps we will examine are always defined on the entire complex plane $\C$.

\begin{definition}[Holomorphic index $\ind{f}{\sigma}$]
    Let $f$ be holomorphic and $\sigma$ a simple fixed point of $f$. Then we call
    \[
        \ind{f}{\sigma} := \frac{1}{1 - f'(\sigma)}
    \]
    the \emph{holomorphic index}.
\end{definition}

Oudkerk's definition of well behaved functions relies on certain properties of the vector fields
\begin{align*}
    \dot z = f(z) - z && \text{and} && \dot z = i[f(z) - z].
\end{align*}
As a continuous-time dynamical system, they serve as a rough approximation for the discrete-time dynamical system of iterated holomorphic functions. This is why we need the following definition.

\begin{definition}[Maximal solution of a vector field, see {\cite[Definition~2.2.1]{oudkerk}}]
    Let $V:D\to\C$ be holomorphic, where $D\subseteq\C$ and let $z_0\in D$. Suppose that $\gamma:I\to\C$ satisfies
    \begin{enumerate}
        \item $I$ is an interval in $\R$, $0\in I$;
        \item $\gamma$ solves $\dot z = V(z)$;
        \item $\gamma(0) = z_0$.
    \end{enumerate}
    We say that $\gamma$ is \emph{maximal} if given any other $\Tilde{\gamma}:\Tilde{I}\to\C$ satisfying (1), (2) and (3) we have $\Tilde{I} \subseteq I$.
\end{definition}

For an integer $q\ge 1$, we now define the repelling directions $z_{k,-} := r_0e^{2\pi i (k-1)/q}$ and attracting directions $z_{k, +} := e^{\pi i/q}z_{k, -}$ for $k \in \Z/q\Z$. Given an $i\in\Z/q\Z$, $s\in\{+, -\}$ and $f\in\Ncal$, we let $\gamma_{i,s,f}:I\to\C$ be the maximal solution of the vector field $\dot z = i[f(z) - z]$ satisfying $\gamma_{i,s,f}(0) = z_{i,s}$. Also let $\ell_{i,s,f} := \gamma_{i,s,f}(I)$.

\begin{definition}[well behaved, $\WB$, see {\cite[Definition~2.3.1]{oudkerk}}]
    We say that $f\in\Ncal$ is \emph{well behaved} if every forward and backward trajectory for the vector field $\dot z = i[f(z) - z]$ passing through the points $z_{i,s}$ stays in the disc $D_{r_0/2}$ once it has entered.

    More specifically, for each $i\in\Z/q\Z$ and $s\in\{+,-\}$ there are some $t_-, t_+\in\R$ such that $t_- < 0 < t_+$ and
    \begin{align*}
        \gamma_{i,s,f}\left((-\infty, t_-)\right) &\subset D_{r_0/2}, \\
        \gamma_{i,s,f}\left([t_-, t_+]\right) &\subset K_0 \setminus D_{r_0/2} & \text{and}\\
        \gamma_{i,s,f}\left((t_+, +\infty)\right) &\subset D_{r_0/2}.
    \end{align*}
    Define $\WB := \{f\in\Ncal \mid \text{$f$ is well behaved}\}$.
\end{definition}

The reason for limiting ourselves to well behaved maps is that their fundamental regions and Fatou coordinates are fairly easy to construct, while also describing the behavior of many interesting maps. In particular, all interesting\footnote{The only non-well behaved maps close to parabolic parameters lie inside of the hyperbolic components, which do not exhibit interesting dynamics for our purposes.} maps $p_c$ with parameters close to some bifurcation $c_0$ are well behaved making the definition highly useful for our discussion of perturbations of $c_0$ along the parameter rays in \Cref{sec:application}.

Let $i\in\Z/q\Z$ and $s\in\{+,-\}$. If $\lim_{t\to\pm\infty}\gamma_{i,s,f}(t)$ exists then we let
\[
    \gamma_{i,s,f}(\pm\infty) := \lim_{t\to\pm\infty}\gamma_{i,s,f}(t).
\]
Both $\gamma_{i,s,f}(+\infty)$ and $\gamma_{i,s,f}(-\infty)$ will be fixed points for $f$ (see \Cref{prop:combinatorics}).

\begin{proposition}[Combinatorics for well behaved maps, see {\cite[Proposition~2.3.2]{oudkerk}}]\label{prop:combinatorics}
    If $f\in\WB$ then the following hold.
    \begin{enumerate}
        \item Every trajectory $\gamma_{i,s,f}(t)$ converges to a fixed point (close to 0) as $t\to\pm\infty$.
        \item For any fixed point $\sigma$ of $f$ in $K_0$, there is some $i\in\Z/q\Z$ and $s\in\{+,-\}$ such that either $\gamma_{i,s,f}(+\infty) = \sigma$ or $\gamma_{i,s,f}(-\infty) = \sigma$.
        \item For all $i\in\Z/q\Z$ we have $\gamma_{i,-,f}(+\infty) = \gamma_{i,+,f}(+\infty)$ and $\gamma_{i,-,f}(-\infty) = \gamma_{i-1,+,f}(-\infty)$.
        \item For each $i\in\Z/q\Z$ and $s\in\{+,-\}$, either the closure of $\ell_{i,s,f}$ is homeomorphic to a circle (in which case $\gamma_{i,s,f}(+\infty) = \gamma_{i,s,f}(-\infty)$ is a multiple fixed point), or there is a unique $j\in\Z/q\Z$ so that the closure of $\ell_{i,s,f}\cup\ell_{j,\overline{s},f}$ is homeomorphic to a circle, where $s\neq\overline{s}\in\{+,-\}$.
        \item For any $i\in\Z/q\Z$ and $s\in\{+,-\}$ we have $\ell_{i,s,f}\cap f(\ell_{i,s,f}) = \varnothing$, and the closure of $\ell_{i,s,f} \cup f(\ell_{i,s,f})$ is a Jordan contour which bounds a closed Jordan domain $S_{i,s,f}$. These $S_{i,s,f}$ can only intersect one another at the fixed points (which lie at their endpoints $\gamma_{i,s,f}(+\infty)$ and $\gamma_{i,s,f}(-\infty)$).
    \end{enumerate}
\end{proposition}

We set $S'_{i,s,f} := S_{i,s,f}\setminus\{\gamma_{i,s,f}(+\infty), \gamma_{i,s,f}(-\infty)\}$. We call these the \emph{fundamental regions} for $f$.

Notice that if $\sigma = \gamma_{i,s,f}(+\infty)$ (for some $i\in\Z/q\Z$ and $s\in\{+,-\}$) then either $\sigma$ is a multiple fixed point, or $\Imag f'(\sigma) >0$ and \enquote{the dynamics of $f$ rotate anti-clockwise around $\sigma$,} i.e.\ $\sigma$ is a sink for the vector field $\dot z = i[f(z) - z]$ (as per \cite[Corollary~3.3.3]{oudkerk}). Similarly, if $\sigma = \gamma_{i,s,f}(-\infty)$ then $\sigma$ is either a multiple fixed point, or $\Imag f'(\sigma) < 0$ and \enquote{the dynamics of $f$ rotate clockwise around $\sigma$,} i.e.\ $\sigma$ is a source for the vector field $\dot z = i[f(z) - z]$.

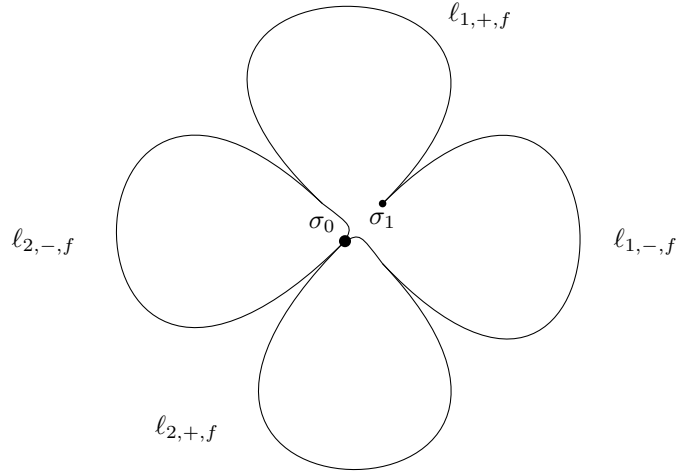
\begin{figure}
    \begin{tikzpicture}
        \draw (0.5,0.5) .. controls (0.6,0.7) .. (0.2,1) .. controls (-3.5,4.5) and (4.5,4.5) .. (1,1);
        \node at (4.5,0.5) {$\ell_{1,-,f}$};
        \draw (0.5,0.5) .. controls (0.7,0.6) .. (1,0.2) .. controls (4.5,-3.5) and (4.5,4.5) .. (1,1);
        \node at (2.3,3.5) {$\ell_{1,+,f}$};
        \draw (0.2,1) .. controls (-3.5,4.5) and (-3.5,-3.5) .. (0.5,0.5);
        \node at (-3.5,0.5) {$\ell_{2,-,f}$};
        \draw (1,0.2) .. controls (4.5,-3.5) and (-3.5,-3.5) .. (0.5,0.5);
        \node at (-1.6,-2) {$\ell_{2,+,f}$};
        \fill (0.5,0.5) circle (0.08) node[anchor=south east] {$\sigma_0$};
        \fill (1,1) circle (0.05) node[anchor=north] {$\sigma_1$};
    \end{tikzpicture}
    \caption{Example of a gate structure for $q=2$. $\sigma_0$ is a double fixed point while $\sigma_1$ is a simple fixed point. $\gate{f} = (1,\star)$, i.e.\ $\overline{\ell_{1,+,f}\cup\ell_{1,-,f}}$ is homeomorphic to a circle and $\overline{\ell_{2,+,f}}$ is homeomorphic to a circle.}
    \label{fig:gate structure example}
\end{figure}

\begin{definition}[Gate structure, $\gate{f}$, see {\cite[Definition 2.3.3]{oudkerk}}]\label{def:gate structure}
    For an $f\in\WB$ we form the vector $\gate{f} = (\gate[1]{f},\dots, \gate[q]{f})$ where
    \[
        \gate[i]{f} := \begin{cases}
            j & \text{if $\overline{\ell_{i,+,f}\cup\ell_{j,-,f}}$ is homeomorphic to a circle;} \\
            \star & \text{if $\overline{\ell_{i,+,f}}$ is homeomorphic to a circle.}
        \end{cases}
    \]
    This is well defined and for $f_0$ we get $\gate{f_0}=(\star,\dots,\star)$.

    The $i$th gate is said to be \emph{open} if $\gate[i]{f}\neq\star$, \emph{closed} if $\gate[i]{f}=\star$.
\end{definition}
See also \Cref{fig:gate structure example}.

\begin{definition}[\textsf{Admissible} and $\WB(G)$, see {\cite[Definition~2.3.4]{oudkerk}}]
    Note that, although every gate structure has an associated vector $G\in\{1,\dots,q,\star\}^q$, not every such vector corresponds to an admissible gate structure.

    Let us draw a circle and place along it points labeled in anti-clockwise order $(1,-)$, $(1,+)$, $(2,-)$, $(2,+)$, \dots, $(q,-)$, $(q,+)$. A vector $G = (G_1,\dots, G_q)\in\{1,\dots,q,\star\}^q$ is called an \emph{admissible} gate structure if for each $j\in Z/q\Z$ there is at most one $i\in Z/q\Z$ such that $G_i = j$, and if we can draw non-intersecting lines on the disk between each pair $(i,+)$, $(j,-)$ for which $G_i=j$.

    Let \textsf{Admissible} be the set of all such \emph{admissible} vectors and $\WB(G) := \{f\in\WB\mid \gate{f} = G\}$.
\end{definition}

Another result from Oudkerk is the following corollary enabling us to easily identify well behaved $f\in\Ncal$.

\begin{corollary}[Sufficient conditions for $f$ to be well behaved, see {\cite[Corollary 3.7.5]{oudkerk}}]\label{cor:sufficient condition}
    Suppose that $f\in\Ncal$ where $\Ncal$ is a small neighborhood of $f_0$, and that $\operatorname{Fix}(f) = \{\sigma\in K_0\mid f(\sigma) = \sigma\}$. There is a constant $M>0$ such that if for every set $X$ with $\varnothing\subsetneq X\subsetneq Fix(f)$ we have
    \[
        \left|\Imag\sum_{\sigma\in X}\ind{f}{\sigma}\right| > M,
    \]
    then $f\in\WB$.
\end{corollary}

\begin{theorem}[Existence and continuity of Fatou coordinates, see {\cite[Theorem~2.3.12]{oudkerk}}]\label{the:existence fatou coords}
    Let $f\in\WB$, $i\in\Z/q\Z$ and $s\in\{+,-\}$.
    \begin{enumerate}
        \item There exists an analytic univalent map $\Phi_{i,s,f}$ defined in a neighborhood of $S'_{i,s,f}$, satisfying
        \begin{align*}
            \Phi_{i,s,f}(f(z)) = \Phi_{i,s,f}(z) + 1 && \text{if $z\in\ell_{i,s,f}$}.
        \end{align*}
        This is unique up to addition by a constant. We call this a \emph{Fatou coordinate} of $f$.
        \item The \emph{Écalle cylinder} $\mathcal{C}_{i,s,f} = S'_{i,s,f}/f$, obtained by identifying $z$ and $f(z)$ for all $z\in\ell_{i,s,f}$, is conformally isomorphic to the cylinder $\C/\Z$ via $[z]_f\mapsto [\Phi_{i,s,f}(z)]_\Z$.
        \item\label{the:ex fatou coords:cont of S} For each $G\in\textsf{Admissible}$, the map $f\mapsto S_{i,s,f}$ is Hausdorff continuous on $\WB(G)$, but only Hausdorff lower semi-continuous on $\WB$ (using the compact-open topology).
        \item There is a normalization of the Fatou coordinates such that $f\mapsto (\Phi_{i,s,f}:S'_{i,s,f}\to\C$ is continuous on $\WB$ (using the compact-open topology on both sides).
    \end{enumerate}
\end{theorem}

There is a \enquote{preferred normalization} of the Fatou coordinates which satisfies \Cref{the:existence fatou coords} (4), see \Cref{the:formula lph}.

\begin{definition}[The sets $U_{i,s,f}$, see {\cite[Definition~2.3.13]{oudkerk}}]
    Let $f\in\WB$. Then if we have $\gate[i]{f}=j\neq\star$ we let $U_{i,+,f} = U_{j,-,f}$ be the open Jordan domains bounded by the closure of $f^{\circ-2}(\ell_{i,+,f})\cup f^{\circ2}(\ell_{j,-,f})$.

    If $\gate[i]{f} = \star$ then we let $U_{i,+,f}$ be the open set bounded by the closure of $f^{\circ-2}(\ell_{i,+f})$.

    And if $\gate[k]{f}\neq j$ for each $k\in\Z/q\Z$ then we let $U_{j,-,f}$ be the open set bounded by the closure of $f^{\circ2}(\ell_{j,-,f})$.

    Notice that we will always have $S'_{i,s,f}\subset U_{i,s,f}$, and that for all $i,j\in\Z/q\Z$ we have $U_{i,+,f}=U_{j,-f}$ if and only if $\gate[i]{f} = j$.
\end{definition}

\begin{proposition}[Extending $\Phi_{i,s,f}$ to $U_{i,s,f}$, see {\cite[Proposition~2.3.14]{oudkerk}}]
    Let $f\in\WB$, $i\in\Z/q\Z$ and $s\in\{+,-\}$.
    \begin{enumerate}
        \item We can extend the Fatou coordinate (defined on $S'_{i,s,f}$) to give an analytic map $\Phi_{i,s,f}:U_{i,s,f}\to\C$ satisfying
        \begin{align*}
            \Phi_{i,s,f}(f(z)) = \Phi_{i,s,f}(z) + 1 && \text{if $z,f(z)\in U_{i,s,f}$},
        \end{align*}
        and $\Phi_{i,s,f}$ is unique up to addition by a constant. Also we have $U_{i,s,f}/f = S'_{i,s,f}/f = \mathcal{C}_{i,s,f}$.
        \item $f\mapsto (\Phi_{i,s,f}:U_{i,s,f}\to\C)$ and $f\mapsto\overline{U_{i,s,f}}$ are continuous on $\WB(G)$ for each $G\in\textsf{Admissible}$. (However, neither is continuous on $\WB$, and $f\mapsto\overline{U_{i,s,f}}$ is not even lower semi-continuous.)
    \end{enumerate}
\end{proposition}

\begin{definition}[Lifted phase, $\lph{i}{f}$, see {\cite[Definition~2.4.1]{oudkerk}}]
    Recall that if $f\in\WB(G)$ and $G_i=j\neq\star$ then $\Phi_{i,+,f}$ and $\Phi_{j,-,f}$ are both defined on $U_{i,+,f}=U_{j,-,f}$ and differ by a constant (since Fatou coordinates are unique up to addition by a constant).

    Therefore the \emph{lifted phase for the $i$th gate}
    \[
        \lph{i}{f} := \begin{cases}
            \Phi_{j,-,f} - \Phi_{i,+,f} & \text{if $j = G_i \neq\star$,} \\
            \infty & \text{if $G_i=\star$},
        \end{cases}
    \]
    is well defined.
\end{definition}

From this, the next proposition almost immediately follows by definition, which is why it is only presented as a remark by Oudkerk.

\begin{proposition}[see {\cite[Remark~2.4.2]{oudkerk}}]\label{prop:central proposition}
    Suppose that $f\in\WB$ and $G_i=j\neq\star$. For each $z\in S'_{i,+,f}$ there is a \emph{smallest} positive integer $n_f$ such that $f^{\circ n}(z)\in U_{i,+,f}$ for all $n\in\{0,\dots,n_f\}$ and $f^{\circ n_f}(z)\in S'_{j,-,f}$. Also
    \[
        \Phi_{j,-,f}\left(f^{\circ n_f}(z)\right) = \Phi_{i,+,f}(z) + \lph{i}{f} + n_f.
    \]
\end{proposition}

\begin{definition}[$\text{Fix}^u_i(f)$ and $\text{Fix}^\ell_i(f)$, see {\cite[Definition~2.4.10]{oudkerk}}]
    Suppose that $f\in\WB$, $i\in\Z/q\Z$ and $\gate[i]{f}\neq\star$. Notice that by definition of $f$ being well behaved, $\overline{D_{r_0/2}}\setminus U_{i,+,f}$ has two components. We denote by $\text{Upper}(i,f)$ the component containing $\gamma_{i,+,f}(+\infty)$, and by $\text{Lower}(i,f)$ the component containing $\gamma_{i,+,f}(-\infty)$.

    We can then decompose the set of fixed points $\text{Fix}(f) := \{\sigma\in K_0\mid f(\sigma) = \sigma\}$ into the disjoint union $\text{Fix}^u_i(f)\sqcup\text{Fix}^\ell_i(f)$ by letting $\text{Fix}^u_i(f) := \text{Upper}(i,f)\cap\text{Fix}(f)$ and $\text{Fix}^\ell_i(f) := \text{Lower}(i,f)\cap\text{Fix}(f)$.
\end{definition}

In \cite[Definition~2.4.7]{oudkerk} the $\jind{f}{\sigma}$ for fixed points $\sigma\in K_0$ of $f\in\Ncal$ is defined. However, we only need the definition for simple fixed points, which is
\[
    \jind{f}{\sigma} := -\frac{2\pi i}{\log f'(\sigma)}
\]
where we take the branch $\Imag\log(\cdot)\in(-\pi,\pi]$ of the logarithm.

\begin{theorem}[Formula for the lifted phases, see {\cite[Theorem~2.4.11]{oudkerk}}]\label{the:formula lph}
    Suppose that we have $f\in\WB(G)$. There is a \emph{preferred normalization} of the Fatou coordinates (so that theorem \ref{the:existence fatou coords} 4. is satisfied) such that if $G_i\neq\star$ then the lifted phase of the $i$th gate is given by
    \[
        \lph{i}{f} = \begin{cases}
            + \sum_{\sigma\in\text{Fix}^u_i(f)}\jind{f}{\sigma} & \text{if $\sigma_0(f)\notin\text{Fix}^u_i(f)$;} \\
            - \sum_{\sigma\in\text{Fix}^\ell_i(f)}\jind{f}{\sigma} & \text{if $\sigma_0(f)\notin\text{Fix}^\ell_i(f)$.}
        \end{cases}
    \]
\end{theorem}

\begin{proposition}[Equivalent convergence criteria, see {\cite[Proposition~2.4.13]{oudkerk}}]\label{prop:convergence criteria}
    Suppose that we have a sequence $(f_k)$ in $\WB(G)$ converging to $f_0$. Then the following are equivalent:
    \begin{enumerate}
        \item $\Real\lph{i}{f_k} \to -\infty$ as $k\to+\infty$ for every $i\in\Z/q\Z$ such that $G_i\neq\star$;
        \item $S_{i,s,f_k}\to S_{i,s,f_0}$ as $k\to +\infty$ for each $i\in\Z/q\Z$ and $s\in\{+,-\}$;
        \item \[
        \overline{U_{i,s,f_k}}\to\begin{cases}
            \overline{U_{i,+,f_0}}\cup\overline{U_{j,-,f_0}} & \text{if $s=+$ and $j = G_i\neq\star$,} \\
            \overline{U_{i,-,f_0}}\cup\overline{U_{j,+,f_0}} & \text{if $s=-$ and $\exists j: G_j=i$,} \\
            \overline{U_{i,s,f_0}} & \text{otherwise}
        \end{cases}
        \] as $k\to +\infty$ for each $i\in\Z/q\Z$ and $s\in\{+,-\}$;
        \item $\Phi_{i,s,f_k}:U_{i,s,f_k}\to\C$ converges to $\Phi_{i,s,f_0}:U_{i,s,f_0}\to\C$ as $k\to+\infty$ for each $i\in\Z/q\Z$ and $s\in\{+,-\}$.
    \end{enumerate}
\end{proposition}

\section{Proof of the main theorem, satellite case}\label{sec:application}

In this section, we will provide a proof of our main theorem: the occurrence of $\pi$ at all satellite bifurcations in $\M$.

Let $c_0$ be a satellite bifurcation from a period $n$ hyperbolic component $H_n$ to a period $qn$ hyperbolic component $H_{qn}$. Then the map $p_{c_0}^{\circ nq}$ is affinely conjugate to $f_0(z) = z + z^{q+1} + O(z^{q+2})$ via some $\varphi_{c_0}(z) = az + b$ with $a,b\in\C$.

Furthermore, two parameter rays at angles $\vartheta_+$ and $\vartheta_-$ land at this satellite bifurcation and can be parameterized as follows (translated by $-c_0$):
\[
    \alpha_\pm:[1,\infty)\to\C,\ \ t\mapsto -c_0 + \begin{cases}
        \Phi^{-1}_\M\left(t\cdot e^{2\pi i\vartheta_\pm}\right) & t >1;\\
        c_0 & t=1.
    \end{cases}
\]

We consider certain small perturbations of the map $p_{c_0}^{\circ qn}$. While there are many conceivable perturbations, the ones that come from nearby parameters in $\M$ are written as $p_{c_0 + \alpha}^{\circ qn}$ with $\alpha\in\C$, $\alpha$ close to $0$; in particular $\alpha=\alpha_{\pm}(t)$ for $t>1$. 

Under these perturbations, the parabolic fixed point $z_0$ with multiplicity $q+1$ splits up into a periodic point of exact period $n$, called $\sigma$, and $q$ periodic points of exact period $qn$, called $\varsigma_k$ for $k=1,2,\dots,q$.

Let us define the map $f_{\alpha_\pm} := \varphi_{c_0}\circ p_{c_0 + \alpha_\pm}^{\circ qn}\circ \varphi_{c_0}^{-1}$. Then of course $f_{\alpha_\pm}\to f_0$ in the compact-open topology as $t\to 1$ or $\alpha_\pm \to 0$. With this, we can establish our first proposition.

\begin{proposition}[$f_{\alpha_\pm}$ is well behaved]\label{prop:well behaved}
    Let $f_{\alpha_\pm(t)}:\DoDef(f_{\alpha_\pm})\to\C$ as defined above. Then there exists some $t_0$ such that for all $t\in [1,t_0)$ the map $f_{\alpha_\pm(t)}$ is well behaved.
\end{proposition}
\begin{proof}
    For $t=1$ the map $f_{\alpha_\pm(1)} = f_0$ and hence is trivially well behaved. For $t>1$ we will apply \Cref{cor:sufficient condition}. To do this, we need to calculate the multipliers of the fixed points of $f_{\alpha_\pm}$ in $K_0$. These will be the $q+1$ simple fixed points described above.
    
    The multipliers of these fixed points are given by the multiplier maps $\mu_n:H_n\to\D$ and $\mu_{qn}:H_{qn}\to\D$. Let us first look at the $q$ fixed points that arise from the period $qn$ periodic point. It is well known that the multiplier map $\mu_{qn}$ can be analytically extended to some neighborhood $U$ of $c_0$. Hence, via the Taylor series expansion of $\mu_{qn}$ at $c_0$ we obtain $\mu_{qn}(c_0 + \alpha_\pm) = \mu_{qn}(c_0) + \mu_{qn}'(c_0)\alpha_\pm + O(\alpha_\pm^2)$ where $\mu_{qn}(c_0) = 1$. This sufficiently describes the multiplier of $\varsigma_k$ for the rest of the proof.

    The multiplier of the remaining fixed point $\sigma$ requires a result from \cite{guckenheimer}. Again, $\mu_n$ can be analytically extended to some neighborhood $U$ of $c_0$ such that we can write $\mu_{n}(c_0 + \alpha_\pm) = \mu_{n}(c_0) + \mu_{n}'(c_0)\alpha_\pm + O(\alpha_\pm^2)$ where $\mu_{n}(c_0)$ is some $q$th root of unity. However, the fixed point's multiplier $f'_{\alpha_\pm}(\sigma)$ will be
    \begin{align*}
        \mu_{n}(c_0 + \alpha_\pm)^q &= \left(\mu_{n}(c_0) + \mu_{n}'(c_0)\alpha_\pm + O(\alpha_\pm^2)\right)^q \\
        &= \left(\mu_{n}(c_0)\left(1 + \overline{\mu_{n}(c_0)}\mu_{n}'(c_0)\alpha_\pm + O(\alpha_\pm^2)\right)\right)^q \\
        &= \left(1 + \overline{\mu_{n}(c_0)}\mu_{n}'(c_0)\alpha_\pm + O(\alpha_\pm^2)\right)^q \\
        &= 1 + q\overline{\mu_{n}(c_0)}\mu_{n}'(c_0)\alpha_\pm + O(\alpha_\pm^2)
    \end{align*}
    and by substituting $\mu_n'(c_0) = -\frac{\mu_{nq}'(c_0)}{q^2\overline{\mu_{n}(c_0)}}$ due to \cite[Theorem~A~(3)]{guckenheimer} we obtain $f'_{\alpha_\pm}(\sigma) = 1 - \mu_{nq}'(c_0)\alpha_\pm/q + O(\alpha_\pm^2)$.

    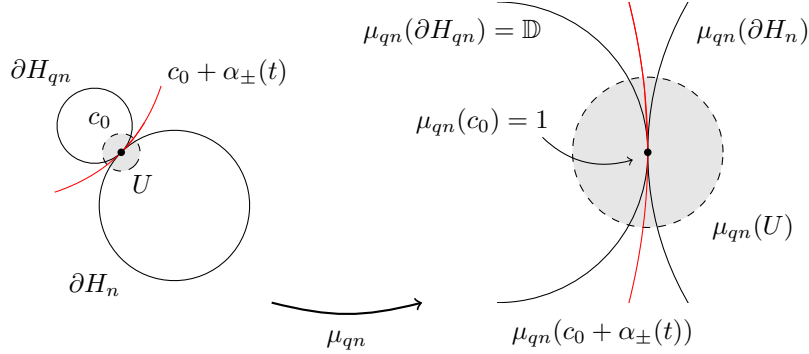
\begin{figure}
        \centering
        \begin{tikzpicture}
        \begin{scope}[shift={(-7,0)}, rotate=-45, scale=0.5]
            \filldraw[dashed, fill=gray!20!white] (0,0) circle (0.5) node at (1,-0.2) {$U$};
            \draw (-1,0) circle (1) node at (-3,0) {$\partial H_{qn}$};
            \draw (2,0) circle (2) node at (2,-3) {$\partial H_n$};
            \draw[domain=-0.5:0,smooth,variable=\x,red] plot ({\x},{sqrt(-8*\x)});
            \draw[domain=-2:1,smooth,variable=\x,red] plot ({-\x*\x/8},{\x});
            \node at (0.5,3.5) {$c_0 + \alpha_\pm(t)$};
            \fill (0,0) circle (0.1) node at (-1,0.2) {$c_0$};
        \end{scope}
        \draw[-To, thick] (-5,-2) .. controls (-4.2,-2.2) and (-3.8,-2.2) .. (-3,-2) node at (-4,-2.5) {$\mu_{qn}$};
        \begin{scope}[scale=2]
            \node at (-1.3, 0.8) {$\mu_{qn}(\partial H_{qn}) = \D$};
            \node at (-0.3, -1.2) {$\mu_{qn}(c_0 + \alpha_\pm(t))$};
            \node at (0.7, 0.8) {$\mu_{qn}(\partial H_{n})$};
            \node at (-1.1, 0.2) {$\mu_{qn}(c_0) = 1$};
            \clip (-1,-1) rectangle (1, 1);
            \filldraw[dashed, fill=gray!20!white] (0,0) circle (0.5) node at (0.7, -0.5) {$\mu_{qn}(U)$};
            \draw (-1,0) circle (1);
            \draw (2,0) circle (2);
            \draw[domain=-0.5:0,smooth,variable=\x,red] plot ({\x},{sqrt(-8*\x)});
            \draw[domain=-2:1,smooth,variable=\x,red] plot ({-\x*\x/8},{\x});
            \fill (0,0) circle (0.025);
            \draw[-To, bend right] (-0.7, 0.1) to (-0.1,-0.05); 
        \end{scope}
        \end{tikzpicture}
        \caption{A neighborhood $U$ of $c_0$ is mapped under $\mu_{qn}$.}
        \label{fig:multiplier map}
    \end{figure}

    Another important detail is that for $\alpha_\pm\to0$ (as $t\to1$) we have
    \[
        \frac{\Real(\mu_{qn}(c_0 + \alpha_\pm) - 1)}{\Imag(\mu_{qn}(c_0 + \alpha_\pm) - 1)} \to 0
    \]
    which intuitively means $|\Imag\mu_{qn}'(c_0)\alpha_\pm| \gg |\Real\mu_{qn}'(c_0)\alpha_\pm|$. Indeed, the boundaries of the two hyperbolic components $H_n$ and $H_{qn}$ are analytic and tangent to each other at $c_0$, and this property is preserved under $\mu_{qn}$, which is conformal in a neighborhood of $c_0$. In the image, the two boundaries are vertical at $\mu(c_0)=1$, so their horizontal distance decreases at least quadratically with respect to the imaginary parts (more detailed estimates can be found in \cite{guckenheimer}, but they are not required for us). 

    Let us now calculate the holomorphic indices of the fixed points. For the $q$ fixed points of the period $qn$ periodic point we get
    \begin{align*}
        \ind{f_{\alpha_\pm}}{\varsigma_k} &= \frac{1}{1- \mu_{qn}(c_0+\alpha_\pm)} = \frac{1}{1 - (1 + \mu_{qn}'(c_0)\alpha_\pm + O(\alpha_\pm^2))} \\
        &= -\frac{1}{\mu_{qn}'(c_0)\alpha_\pm(1 + O(\alpha_\pm))} = -\frac{1}{\mu_{qn}'(c_0)\alpha_\pm} + O(1)
    \end{align*}
    while for the single fixed point, we get
    \begin{align*}
        \ind{f_{\alpha_\pm}}{\sigma} &= \frac{1}{1- \mu_{n}(c_0 + \alpha_\pm)^q} = \frac{1}{1 - (1 - \mu_{qn}'(c_0)\alpha_\pm/q + O(\alpha_\pm^2))} \\
        &= \frac{q}{\mu_{qn}'(c_0)\alpha_\pm(1 + O(\alpha_\pm))} = \frac{q}{\mu_{qn}'(c_0)\alpha_\pm} + O(1).
    \end{align*}
    Now it is obvious that this satisfies the condition in \Cref{cor:sufficient condition}: we can choose some $t_0$ such that for any proper subset $X$ of these fixed points and all $t < t_0$ we satisfy
    \[
        \left|\Imag\sum_{\sigma\in X}\ind{f_{\alpha_\pm}}{\sigma}\right| > M
    \]
    for some fixed $M\in\R$. Hence, $f_{\alpha_\pm}$ is well behaved.
\end{proof}

The indices $\ind{f_{\alpha_+(t)}}{\varsigma_k}$ and $\ind{f_{\alpha_-(t)}}{\varsigma_k}$ have the property that, as $t\to 1$, their imaginary parts diverge, while their sum converges to a finite limit. Therefore, for small $t$, the imaginary parts have different signs, so we can distinguish $\alpha_+(t)$ from $\alpha_-(t)$ by requiring that we have $\Imag \ind{f_{\alpha_+}}{\varsigma_k}>0$ and $\Imag \ind{f_{\alpha_-}}{\varsigma_k}<0$ respectively. 

Now we can easily determine the gate structure of $f_{\alpha_\pm}$.
As a reminder, $\gate{f_{\alpha_+}} = (1,2,3\dots,q)$ corresponds to attracting petal $1$ being connected to repelling petal $1$ and so forth, while $\gate{f_{\alpha_-}} = (2,\dots,q,1)$ means attracting petal $1$ is connected to repelling petal $2$ and so forth and lastly attracting petal $q$ is connected to repelling petal $1$ (see \Cref{fig:gate structure}).

\begin{proposition}\label{prop:gate structure}
    $f_{\alpha_+(t)}\in\WB((1,2,3,\dots,q))$ and $f_{\alpha_-(t)}\in\WB((2,3,4,\dots,q,1))$ for $t>1$.
\end{proposition}

\begin{proof}
    Because of \Cref{prop:well behaved}, $f_{\alpha_\pm}$ is well behaved and we can construct $\gamma_{i,s,f_{\alpha_\pm}}$ where $i\in\Z/q\Z$ and $s\in\{+,-\}$ as described in \Cref{sec:oudkerk}. Then by \Cref{prop:combinatorics} (1) every such trajectory $\gamma_{i,s,f_{\alpha_\pm}}(t)$ converges to one of the fixed points $\sigma$, $\varsigma_k$ as $t\to\pm\infty$. Moreover, due to \Cref{prop:combinatorics} (3) we know that for all $i\in\Z/q\Z$
    \begin{align*}
        \gamma_{i,-,f_{\alpha_\pm}}(+\infty) = \gamma_{i,+,f_{\alpha_\pm}}(+\infty) && \gamma_{i,-,f_{\alpha_\pm}}(-\infty) = \gamma_{i-1,+,f_{\alpha_\pm}}(-\infty)
    \end{align*}
    Since all fixed points are simple for $\alpha_\pm\neq0$, we also know that never $\gamma_{i,s,f_{\alpha_\pm}}(+\infty) = \gamma_{i,s,f_{\alpha_\pm}}(-\infty)$ and therefore that all gates are open.

    Let us now consider $f_{\alpha_+}$ by itself. As defined above, we have
    \begin{align*}
        \Imag \ind{f_{\alpha_+}}{\varsigma_k}>0 && \text{and therefore also} && \Imag \ind{f_{\alpha_+}}{\sigma}<0,
    \end{align*}
    which means $\sigma$ is a source and $\varsigma_k$ are sinks of the vector field $\dot z = i[f(z) - z]$. Moreover, because of \Cref{prop:combinatorics} (2) each fixed point must be at least at the end of one of the trajectories $\gamma_{i,s,f_{\alpha_+}}$. This implies
    \begin{align*}
        \sigma = \gamma_{i,s,f_{\alpha_+}}(-\infty) & \quad \text{and} \\
        \varsigma_i = \gamma_{i,s,f_{\alpha_+}}(+\infty) & \quad \text{for all $i\in\Z/q\Z$, $s\in\{+,-\}$},
    \end{align*}
    since there are $q$ fixed points $\varsigma_i$ and $q$ trajectories $\gamma_{i,s,f_{\alpha_+}}(+\infty)$. All in all, this implies that the only possible gate structure $G\in\textsf{Admissible}$ is $G=(1,2,\dots,q)$.
    The same reasoning can be applied to $f_{\alpha_-}$ yielding $\gate{f_{\alpha_-}} = (2,3,4,\dots,q,1)$.
\end{proof}

The two gate structures of $f_{\alpha_+}$ and $f_{\alpha_-}$ for $q=3$ can be seen \Cref{fig:gate structure}.

\newcommand{\arrow}[2]{
\begin{scope}[rotate=#2, shift={(2,0)}]
\begin{scope}[xscale=#1, yscale=1]
    \draw[-To] (-0.3,0) -- (0.3,0);    
\end{scope}
\end{scope}}

\newcommand{\reppetal}[2]{
\begin{scope}[rotate=#1] 
    \draw (0,0) .. controls (4,-2.6) and (4,2.5) .. (1, 0.57) node at (4,0) {#2};
\end{scope}}

\newcommand{\attpetal}[2]{
\begin{scope}[xscale=1, yscale=-1]
    \reppetal{-#1}{#2};
\end{scope}}

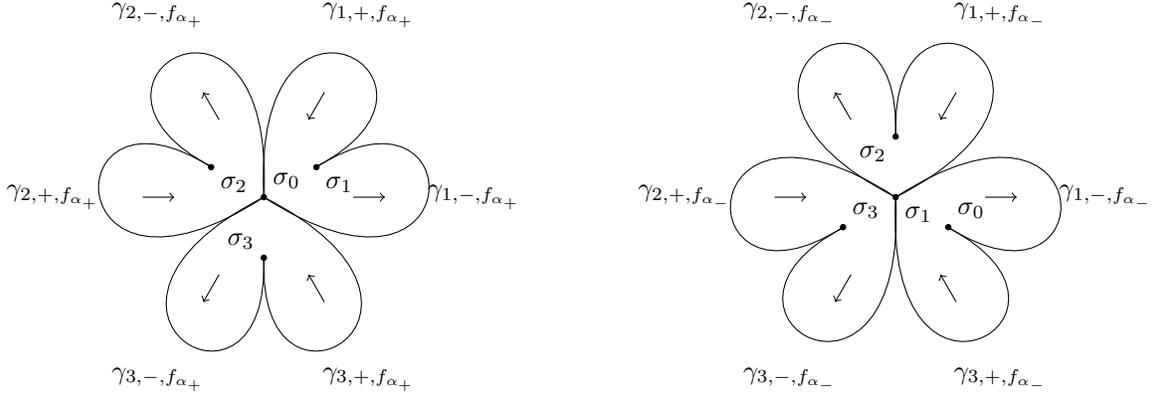
\begin{figure}
    \centering
    \begin{tikzpicture}
    \begin{scope}[scale=0.7]
        \reppetal{0}{$\gamma_{1,-,f_{\alpha_+}}$}
        \arrow{1}{0}
        \reppetal{120}{$\gamma_{2,-,f_{\alpha_+}}$}
        \arrow{1}{120}
        \reppetal{240}{$\gamma_{3,-,f_{\alpha_+}}$}
        \arrow{1}{240}
        \attpetal{60}{$\gamma_{1,+,f_{\alpha_+}}$}
        \arrow{-1}{60}
        \attpetal{180}{$\gamma_{2,+,f_{\alpha_+}}$}
        \arrow{-1}{180}
        \attpetal{300}{$\gamma_{3,+,f_{\alpha_+}}$}
        \arrow{-1}{300}
        \filldraw (0,0) circle (0.05) node[anchor=south west] {$\sigma_0$};
        \filldraw (1,0.57) circle (0.05) node[anchor=north west] {$\sigma_1$};
        \filldraw (-1,0.57) circle (0.05) node[anchor=north west] {$\sigma_2$};
        \filldraw (0,-1.151) circle (0.05) node[anchor=south east] {$\sigma_3$};
        \begin{scope}[shift={(12, 0)}]
            \attpetal{0}{$\gamma_{1,-,f_{\alpha_-}}$}
            \arrow{1}{0}
            \attpetal{120}{$\gamma_{2,-,f_{\alpha_-}}$}
            \arrow{1}{120}
            \attpetal{240}{$\gamma_{3,-,f_{\alpha_-}}$}
            \arrow{1}{240}
            \reppetal{60}{$\gamma_{1,+,f_{\alpha_-}}$}
            \arrow{-1}{60}
            \reppetal{180}{$\gamma_{2,+,f_{\alpha_-}}$}
            \arrow{-1}{180}
            \reppetal{300}{$\gamma_{3,+,f_{\alpha_-}}$}
            \arrow{-1}{300}
            \filldraw (0,0) circle (0.05) node[anchor=north west] {$\sigma_1$};
            \filldraw (1,-0.57) circle (0.05) node[anchor=south west] {$\sigma_0$};
            \filldraw (-1,-0.57) circle (0.05) node[anchor=south west] {$\sigma_3$};
            \filldraw (0,1.151) circle (0.05) node[anchor=north east] {$\sigma_2$};
        \end{scope}
    \end{scope}
    \end{tikzpicture}
    \caption{Gate structures of $f_{\alpha_+}$ and $f_{\alpha_-}$ for $q=3$.}
    \label{fig:gate structure}
\end{figure}

Due to an algorithm of Oudkerk on page 14 in \cite{oudkerk} the fixed points can be relabeled in the following way:
\begin{itemize}
	\item for $\alpha_+$, we set  $\sigma_0 := \sigma$ and $\sigma_i := \varsigma_i$,  where $1\leq i \leq q$ 
	\item for $\alpha_-$, we set $\sigma_0 := \gamma_{1,+,f_{\alpha_-}}(-\infty)$, $\sigma_1 := \sigma = \gamma_{1,+,f_{\alpha_-}}(+\infty)$ and $\sigma_i := \gamma_{i,+,f_{\alpha_-}}(-\infty)$,  where $2\leq i \leq q$. 
\end{itemize}
 These labels will be important when we calculate the lifted phases (see \Cref{the:formula lph}).

We have now established that small parameters $\alpha_\pm$ on parameter rays of a satellite bifurcation $c_0$ are well behaved, and we have determined their gate structures. 

 Our next goal is to approximate the number of iterations that the orbit $(z_n)$ of the critical point $z_0 = 0$ of the map $p_{c_0+\alpha_\pm}$ takes when starting in some attracting petal, iterating through some gate and reaching some repelling petal, until the escape radius $R$ is reached, i.e.\ $|z_n|>R$. 
 To do this, we will approximate the number of iterations for the corresponding map $f_{\alpha_\pm}$ of points starting in some starting region $K_+$ and reaching some landing region $K_-$. Then \Cref{prop:central proposition} makes it possible to relate this number to the multipliers of the fixed points which depend on the perturbation $\alpha_{\pm}$. This is what our next lemma does. 
 
 But first, we need the following definition:
\begin{definition}[Horn property]
    Let $i\in\Z/q\Z$, $G_i\neq\star$ and $K_+\subset S'_{i,+,f_0}$ be compact. 
    We say $f\in\WB(G)$ \emph{satisfies the horn property} for a given $K_+$, a gate $i$ with $G_i=j\neq \star$, and a distance $\delta>0$ if for all $z\in K_+$, we have
        \[
    f^{\circ n_f}(z) \in S'_{j,-,f} \setminus (D_\delta(\sigma_+) \cup D_\delta(\sigma_-))
    \]
    where $\sigma_\pm:=\gamma_{i,+,f}(\pm\infty)$ are the two fixed points bounding the gate, and  $n_f=n_f(z)$ is the integer given by Proposition~\ref{prop:central proposition} (the transit time).

    We denote by $\WB(G, i, K_+,\delta)$ the set of all $(f,\DoDef(f))\in\WB(G)$ that satisfy the horn property for the $i$th gate with given $K_+$ and $\delta$.
\end{definition}

The meaning of the horn property is that all $z\in K_+$, after passing the gate and reaching the fundamental region $S'_{j,-,f}$ of some repelling petal, stay uniformly away from the fixed points at the two ends of $S'_{j,-,f}$.  

The following lemma is a generalization of \cite[Theorem~1.2]{paul}. In this lemma, we fix a gate structure $G$, a gate $i\in\Z/q\Z$ with $j:=G_i  \neq \star$,  a compact starting region $K_+\subset S'_{i,+,f_0}$, and a distance $\delta>0$. For $f\in\WB(G,i,K_+,\delta)$, translate coordinates so that the fixed point $\gamma_{i,+,f}(-\infty) = 0$. Define the landing region $K_- := S'_{j,-,f_0}\setminus D_{\delta/2}(0)$. For $z_0\in K_+$ define $N_f(z_0):=\min\{k\in\N\mid f^{\circ k}(z_0)\in K_-\}$ (the minimal number of iterations until the orbit of $z_0$ enters $K_-$). 

\begin{lemma}\label{lem:bound}
	For $f\to f_0$ in $f\in\WB(G,i,K_+,\delta)$ satisfying \Cref{prop:convergence criteria}~(1) there exists a $C > 0$ such that for all $z\in K_+$ we have $|N_f(z) - \lph{i}{f}| < C$ as $f\to f_0$.
\end{lemma}

\begin{proof}
    For $f\in\WB(G,i,K_+,\delta)$ and $z_0\in K_+$ we have some integer $n_f$ such that $f^{\circ n_f}(z_0) =: z' \in S_{j,-,f}$ by proposition \ref{prop:central proposition}. Our goal now is to show that $n_f\approx N_f$. Let $\textbf{S}_- := f_0^{-1}(S_{j,-,f_0})\cup S_{j,-,f_0} \cup f_0(S_{j,-,f_0})$. Let $U\subset \textbf{S}_-\cup D_\delta(0)$ be some $\varepsilon$-neighborhood of $S_{j,-,f_0}$. Moreover, by taking $f$ sufficiently close to $f_0$, we have $z'\in U \setminus D_\delta(0)$ because $S_{j,-,f}\to S_{j,-,f_0}$ by \Cref{prop:convergence criteria}~(2) and therefore eventually $S_{j,-,f}\subset U$. This means that one of $f_0^{-1}(z'), z',f_0(z')\in S_{j,-,f_0}$ but none lie in $D_{\delta/2}(0)$. By taking $f$ sufficiently close to $f_0$ on $\textbf{S}_-\setminus D_\delta(0)$, we can ensure that one of $f^{-1}(z'), z',f(z')\in K_-$. This means $|N_f - n_f| \leq 1$. Also see \Cref{fig:bounding N} for a sketch of the situation.

    Observe now that the values $\Phi_{i,+,f}(z_0)$ and $\Phi_{j,-,f}(f^{\circ n_f}(z_0))$ in the equation from \Cref{prop:central proposition} converge as $f\to f_0$. For $\Phi_{i,+,f}(z_0)$ this is obvious due to \Cref{prop:convergence criteria}~(4) and for $\Phi_{j,-,f}(f^{\circ n_f}(z_0))$ we know that $f^{\circ n_f}(z_0) \in f^{-1}(K_-) \cup K_- \cup f(K_-) =: \textbf{K}_-$. Taking $f$ sufficiently close to $f_0$ we can ensure that $\textbf{K}_-$ lies in $U_{j,-,f}$ and $U_{j,-,f_0}$. Then $\Phi_{j,-,f}$ and $\Phi_{j,-,f_0}$ are defined on $\textbf{K}_-$, hence $|\Phi_{j,-,f}(f^{\circ n_f}(z_0))| \leq \sup_{z\in\textbf{K}_-}|\Phi_{j,-,f}(z)|$. Thus as $f\to f_0$ both $\Phi_{i,+,f}(z_0)$ and $\Phi_{j,-,f}(f^{\circ n_f}(z_0))$ converge and therefore
    \[
        |N_f - |\lph{i}{f}|| \leq |N_f + \lph{i}{f}| < C
    \]
    is bounded by some real constant $C$.
\end{proof}

\begin{figure}[h]
    \centering
    \begin{tikzpicture}
        \filldraw[fill=gray!20!white, draw=black]
            (0,0) .. controls (-3,3.3) and (3,3.3) ..
            (0,0) .. controls (6.5,6) and (-6.5,6) .. cycle
            node at (2.5, 3.5) {$\textbf{S}_{-}$};
        \filldraw[fill=red!20!white, draw=red!20!white, opacity=1,line width=7pt,rounded corners]
            (0,0) .. controls (-3.5,4) and (3.5,4) ..
            (0,0) .. controls (5,5) and (-5,5) .. cycle;
        \filldraw[fill=gray!20!white, draw=black]
            (0,0) .. controls (-3.5,4) and (3.5,4) ..
            (0,0) .. controls (5,5) and (-5,5) .. cycle
            node at (0,3.3) {$S_{j,-,0}$};
        \fill[fill=blue!30!white, opacity=0.1]
            (-0.15,0) .. controls (-3.5,4) and (3.5,4) ..
            (0,0) .. controls (5,5) and (-5,5) .. cycle;
        \draw
            (-0.15,0) .. controls (-3.5,4) and (3.5,4) ..
            (0,0) .. controls (5,5) and (-5,5) .. cycle
            node at (-2.2,2) {$S_{j,-,f}$};
        \filldraw (-0.15, 0) circle (0.05) node[anchor=north east] {$\sigma(f)$};
        \filldraw (0,0) circle (0.05) node[anchor=north west] {0};
        \draw[dashed] (0,0) circle (2) node at (2,-1.8) {$D_{\delta}(0)$};
    \end{tikzpicture}
    \caption{Taking $f$ sufficiently close to $f_0$ we can ensure that $|N_f - n_f|\leq 1$. Here, $S_{j,-,f}$ is sufficiently close to $S_{j,-,f_0}$ such that it is contained in an $\varepsilon$-neighborhood $U$ of $S_{j,-,f_0}$ (in red), which in turn is contained in $\textbf{S}_-\cup D_\delta(0)$. $S_{j,-,f}$ ends at 0 and some fixed point $\sigma(f)$ distinct from 0, while both ends of $S_{j,-,f_0}$ land at 0.}
    \label{fig:bounding N}
\end{figure}
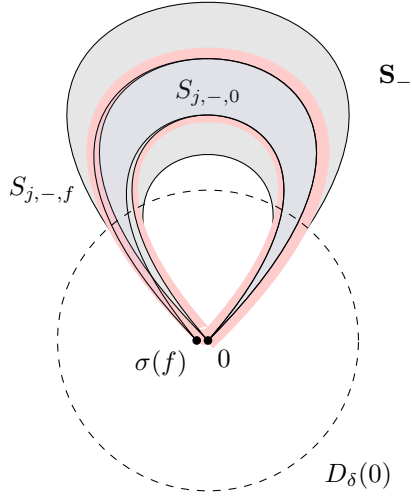

Now we can finally state and prove our first main theorem. This is a more precise reformulation of Theorem~\ref{Thm:One} in the satellite case. 

\begin{theorem}[$\pi$ in the Mandelbrot set, satellite case]\label{the:main theorem}
    Let $c_0$ be a satellite bifurcation from period $n$ to period $qn$. Let $\alpha_\pm(t)$ be defined as above and $N(\alpha_\pm(t)) := \min\{k\in\N\mid |p_{c_0+\alpha_\pm(t)}^{\circ k}(0)|>2\}$. Then for $t\to1$
    \begin{equation}\label{eq:main theorem}
        N(\alpha_\pm(t)) = \pi\cdot \frac{2qn}{|\mu_{qn}'(c_0)\alpha_\pm(t)|} + O(1)
        \;
    \end{equation}
\end{theorem}
\begin{proof}
    We will apply Oudkerk's theory to this situation by again examining $f_{\alpha_\pm} := \varphi_{c_0} \circ p_{c_0+\alpha_\pm}^{\circ qn}\circ\varphi_{c_0}^{-1}$. The critical point $z_0 = 0$ of the maps $p_{c}$ gets mapped to $\varphi_{c_0}(0)$ in these coordinates. First of all, the critical value $p_{c_0+\alpha_\pm}(0)$ lies per definition of the parameter rays on the dynamic rays of $p_{c_0+\alpha_\pm}$ at angles $\vartheta_\pm$ and hence the whole orbit will escape to $\infty$ on these.
    
    We additionally need to know whether the critical orbit when viewed under the perturbed map $p_{c_0+\alpha_\pm}^{\circ nq}$ (or equivalently under $f_{\alpha_\pm}$) will enter through some attracting petal $i\in\Z/q\Z$ and exit through some neighboring repelling petal $j\in\Z/q\Z$ without entering any other petals $U_{k,s,f_{\alpha_+}}$ with $k\in\Z/q\Z$, $s\in\{+,-\}$ and $i\neq k\neq j$ if the perturbation $\alpha_\pm$ lies on the parameter rays of $c_0$. However, this is evident because the dynamic ray $\vartheta_\pm$ is of exact period $q$ when viewed under $p_{c_0+\alpha_\pm}^{\circ n}$ and its images are permuted between the $q$ repelling petals. This means that under $p_{c_0+\alpha_\pm}^{\circ nq}$ the images are fixed in one repelling petal and that each image must exhibit the same behavior as any other image of the ray, i.e.\ if one image passes through two gates, so do the others. Furthermore, it is known that the dynamic rays land at the (images under $p_{c_0+\alpha_\pm}$ of the) critical point before breaking up for parameters $c_0 + \alpha_\pm$ on the parameter rays of $c_0$. Hence, every such image of the dynamic ray will pass through at least one gate to reach an attracting petal containing the images of the critical point. Because the images of the dynamic ray must not intersect by definition and because each image of the dynamic ray must pass through the same number of gates as every other image, it then is under consideration of the gate structure easy to see that the images of the dynamic ray will only pass through at most one gate before reaching an image of the critical point. Therefore, the critical orbit only passes through a single gate from one attracting petal to one repelling petal under the map $p_{c_0+\alpha_\pm}^{\circ nq}$ before diverging to $\infty$.
    
    Thus we can fix some $K_+\subset S'_{i,+,f_0}$ such that for $\alpha_\pm$ sufficiently close to $0$ (or $t$ close to $1$) the critical orbit $f_{\alpha_\pm}^{\circ k(\alpha_\pm)}(\varphi_{c_0}(0))\in K_+$ after some iterations $k(\alpha_\pm) = O(1)$ of constant order. Moreover, since the dynamic ray depends continuously on the parameter in a repelling petal there must exist some $\delta>0$ such that the horn property is satisfied again for $\alpha_\pm$ close to $0$. Additionally, we can see in \Cref{eq:lifted phase} that \Cref{prop:convergence criteria} (1) is satisfied. Thus, we can apply \Cref{lem:bound} to obtain that $N_{f_{\alpha_\pm}} := \min\{n\in\N\mid f_{\alpha_\pm}^{\circ (k(\alpha_\pm) + n)}(\varphi_{c_0}(0))\in K_-\}$ where $K_- := S'_{j,-,f_0}\setminus D_{\delta/2}(\sigma_0)$ and
    \[
        |N_{f_{\alpha_\pm}} - |\lph{i}{f_{\alpha_\pm}}|| < C.
    \]
    After this, the critical again takes some further iterations $\ell(\alpha_\pm) = O(1)$ until it lies one the side connecting to $\infty$ of the Jordan curve $\varphi_{c_0}(\{z\in\C\mid |z| = 2 \})$. Hence, $N(\alpha_\pm) = qnN_{f_{\alpha_\pm}} + O(1)$.

    However, let us now turn to calculate $\lph{i}{f_{\alpha_\pm}}$ which can be done with \Cref{the:formula lph} and under consideration of the gate structure from \Cref{prop:gate structure}. Notice that for $f_{\alpha_+}$ the fixed point $\sigma_0 = \gamma_{i,+,f_{\alpha_+}}(-\infty)$ for all $i\in\Z/q\Z$ and since each petal is connected to its immediate neighboring petal only $\sigma_i\in\text{Fix}^u_i(f)$. It follows that
    \begin{align}
        \lph{i}{f_{\alpha_+}} &= \jind{f_{\alpha_+}}{\sigma_i} = -\frac{2\pi i}{\log f'_{\alpha_+}(\sigma_i)} = -\frac{2\pi i}{\log(1 + \mu_{qn}(c_0)\alpha_+ + O(\alpha_+^2)} \nonumber \\
        &= -\frac{2\pi i}{\mu_{qn}(c_0)\alpha_+ (1+ O(\alpha_+))} = -\frac{2\pi i}{\mu_{qn}(c_0)\alpha_+}(1+ O(\alpha_+)) \nonumber \\
        &= -\frac{2\pi i}{\mu_{qn}(c_0)\alpha_+} + O(1). \label{eq:lifted phase}
    \end{align}
    Similarly for $f_{\alpha_-}$ we obtain $\lph{i}{f_{\alpha_-}} = -\frac{2\pi i}{\mu_{qn}(c_0)\alpha_+} + O(1)$.

    From this it follows that $N_{f_{\alpha_\pm}} = \frac{2\pi}{|\mu_{qn}(c_0)\alpha_\pm|} + O(1)$ and thus
    \begin{equation*}
        N(\alpha_\pm(t)) = \pi\cdot\frac{2qn}{|\mu_{qn}(c_0)\alpha_\pm(t)|} + O(1)
    \end{equation*}
    as $t\to1$.
\end{proof}

This completes the proof of the occurrence of $\pi$ at all satellite bifurcations in $\M$.

From this we also obtain the non-trivial corollary that the escape rates along both parameter rays of a satellite bifurcation are equal.
\begin{corollary}[Rates of escape along parameter rays]
    Let $c_0$ be a satellite bifurcation and let $\alpha_+(t)$ and $\alpha_-(t)$ be parameterized as above. Then for some $t_+,t_-\in\R$ with  $|\alpha_+(t_+)| = |\alpha_-(t_-)|$ we have $N(\alpha_+(t_+)) = N(\alpha_-(t_-)) + O(1)$.
\end{corollary}

\begin{remark}[Conceptual explanation]\label{rem:conceptual}
    \Cref{the:main theorem} also comes with a conceptual explanation similar to the one given in \cite{paul}. This is illustrated in \Cref{fig:conceptual} for the case of $f_{\alpha_+}$ and $q=2$. The curves $\gamma_{i,s,f_{\alpha_+}}$ tend to the fixed points tangent to a certain angle at the fixed points. If we consider the images of $\gamma_{2,+,f_{\alpha_+}}$ under $f_{\alpha_+}$, we can see on the one hand how they match the critical orbit $(z_n)$ iterating through the gate formed by the fixed points $\sigma_0$ and $\sigma_2$. On the other hand, the tangent angle of these curves at the fixed point $\sigma_2$ changes each iteration by a tiny amount, by $\arg f'_{\alpha_+}(\sigma_2) = \Imag\log f'_{\alpha_+}(\sigma_2)$ to be precise. This is due to $f_{\alpha_+}$ being conjugate to $z\mapsto z\cdot f'_{\alpha_+}(\sigma_2)$ in some small neighborhood around $\sigma_2$. The clou is now that, as indicated in \Cref{fig:conceptual}, the total rotation at $\sigma_2$ that $\gamma_{2,+,f_{\alpha_+}}$ has to take until the critical orbit reaches $\gamma_{2,-,f_{\alpha_+}}$ is exactly $2\pi$. Similarly, at $\sigma_0$ the total rotation is exactly $\pi$. Hence we can estimate the number of iterations $N(\alpha_+)$ that the critical orbit takes with
    \[
        N(\alpha_+) \approx \frac{2\pi}{\arg f'_{\alpha_+}(\sigma_2)} \approx \frac{\pi}{\arg f'_{\alpha_+}(\sigma_0) } \approx -\lph{2}{f_{\alpha_+}}.
    \]
    This conceptual understanding which we just outlined for $f_{\alpha_+}$ and $q=2$ also holds in general for all other bifurcations and provides a strong intuition behind the proof. It also gives a satisfying explanation of why $\pi$ appears at all bifurcations in the Mandelbrot set.
\end{remark}

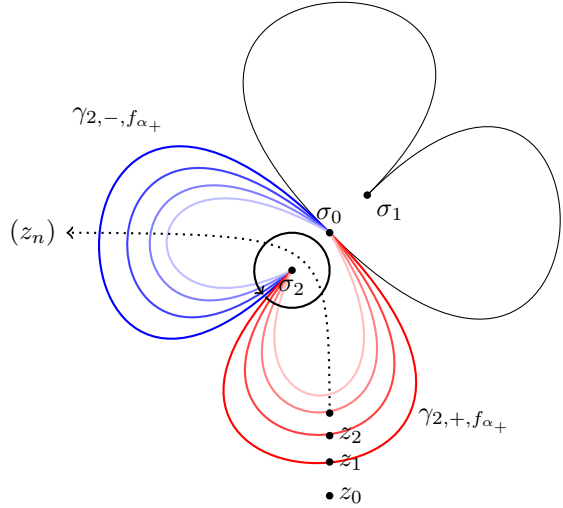
\begin{figure}
    \centering
    \begin{tikzpicture}
        \draw (0.5,0.5) .. controls (-3.5,4.5) and (4.5,4.5) .. (1,1);
        \draw (0.5,0.5) .. controls (4.5,-3.5) and (4.5,4.5) .. (1,1);
        \draw[thick,draw=blue] (0.5,0.5) .. controls (-3.5,4.5) and (-3.5,-3.5) .. (0,0) node at (-2.3,2) {$\gamma_{2,-,f_{\alpha_+}}$};
        \draw[thick,draw=blue!75!white] (0.5,0.5) .. controls (-3,3.5) and (-3,-2.5) .. (0,0);
        \draw[thick,draw=blue!50!white] (0.5,0.5) .. controls (-2.6,2.7) and (-2.6,-1.7) .. (0,0);
        \draw[thick,draw=blue!25!white] (0.5,0.5) .. controls (-2.3,2.1) and (-2.3,-1.1) .. (0,0);
        
        \draw[thick,draw=red] (0.5,0.5) .. controls (4.5,-3.5) and (-3.5,-3.5) .. (0,0) node at (2.3,-2) {$\gamma_{2,+,f_{\alpha_+}}$};
        \draw[thick,draw=red!75!white] (0.5,0.5) .. controls (3.5,-3) and (-2.5,-3) .. (0,0);
        \draw[thick,draw=red!50!white] (0.5,0.5) .. controls (2.7,-2.6) and (-1.7,-2.6) .. (0,0);
        \draw[thick,draw=red!25!white] (0.5,0.5) .. controls (2.1,-2.3) and (-1.1,-2.3) .. (0,0);
        \fill (0,0) circle (0.05) node[anchor=north] {$\sigma_2$};
        \fill (1,1) circle (0.05) node[anchor=north west] {$\sigma_1$};
        \fill (0.5,0.5) circle (0.05) node[anchor=south] {$\sigma_0$};
        \draw [->, thick, rotate=45,draw=black] (-0.5,0) arc [start angle=-180, end angle=175,radius=0.5];

        \fill (0.5,-3) circle (0.05) node[anchor=west] {$z_0$};
        \fill (0.5,-2.55) circle (0.05) node[anchor=west] {$z_1$};
        \fill (0.5,-2.2) circle (0.05) node[anchor=west] {$z_2$};
        \fill (0.5,-1.9) circle (0.05);
        \draw[->, thick,dotted] (0.5,-1.8) .. controls (0.48,0.48) .. (-3,0.5) node[anchor=east] {$(z_n)$};
    \end{tikzpicture}
    \caption{Sketch of the conceptual understanding in the case of $f_{\alpha_+}$ and $q=2$. The critical orbit $(z_n)$ passing through the gate corresponds to a rotation of $2\pi$ around the fixed point $\sigma_2$.}
    \label{fig:conceptual}
\end{figure}

\begin{theorem}[$\pi$ in the Mandelbrot set, primitive case]
    Let $c_0$ be a primitive bifurcation of period $n$. Let $\alpha_\pm(t)$ parameterize the parameter rays as above and $N(\alpha_\pm(t)) := \min\{k\in\N\mid |p_{c_0+\alpha_\pm(t)}^{\circ k}(0)| > 2\}$. Then for $t\to1$
    \begin{equation}
        N(\alpha_\pm(t)) = \pi \cdot \frac{2n}{|\mu'_n(c_0)|\cdot|\alpha_\pm(t)|^{1/2}} + O(1).
    \end{equation}
\end{theorem}
We will omit the proof in this case as after the well-known coordinate change $\lambda(c) = \sqrt{c - c_0}$ we have the same situation as with the satellite case: the hyperbolic component of primitive type becomes two components of equal period $n$ whose boundaries are analytic circles tangent at $\lambda(c_0) = 0$. Now the proof is analogous to the satellite case, just that $q=1$, i.e.\ the parabolic periodic point is a fixed point with multiplicity 2 under the map $p_{c_0}^{\circ n}$. When we perturb along the parameter rays, this fixed point splits up into two, whose multipliers are given by the respective multiplier maps of the two hyperbolic components in the new coordinates $\lambda(c)$. The critical orbit then passes through the gate formed by these two fixed points starting in the single attracting petal end and emerging in the single repelling petal before diverging. The coordinate change $\lambda(c)$ also explains the different rate of convergence $|\alpha_\pm|^{1/2}$ for a given $|\alpha_\pm|$.

\section{Numerical experiments}
Since the phenomenon originally was discovered during numerical experiments at parameters $-3/4$, $1/4$, and $-5/4$, we want to present similar experiments for other bifurcations, demonstrating the result from \Cref{the:main theorem}. However, \Cref{the:main theorem} requires approaching bifurcations along their parameter rays which are difficult to calculate accurately. Moreover, original experiments approached the parameter $-3/4$ along the straight vertical line $\alpha(t) = it$ which is not the parameter ray but seems to serve as a sufficient approximation. This motivates \Cref{Thm:Two} which answers the question of how close a sequence $\tilde{c}_n$ must lie to the parameter ray such that the numerical phenomenon can still be observed. We will now provide its proof.

\begin{proof}[Proof of \Cref{Thm:Two}]
    The proof uses the Koebe quarter theorem which can be applied to the biholomorphic map $\Phi_\M\colon \C\setminus\M\to\C\setminus\D$. Fix some $c_n$ and take the open disk $D_R(c_n)$ with maximal radius $R$ such that $D_R(c_n)$ does not intersect $\M$. Let $h(c) := \Phi_\M(R\cdot c + c_n)$ which satisfies $h(\D) = \Phi_\M(D_R(c_n))$. Also let $\beta := \Phi_\M(c_n)$. Then by the Koebe quarter theorem $h(\D)$ contains the disk $D_{|h'(0)|/4}(h(0)) = D_{R\beta/4}(\Phi_\M(c_n))$. Let $g(c) = \Phi_\M^{-1}(\frac{R\beta}{4}\cdot c + \Phi_\M(c_n))$ which satisfies $g(\D) = \Phi_\M^{-1}(D_{R\beta/4}(\Phi_\M(c_n)))$.  We can then apply Koebe again in the other direction yielding
    \[
        D_{|g'(0)|/4}(g(0)) = D_{R/8}(c_n) \subset \Phi_\M^{-1}(D_{R\beta/4}(\Phi_\M(c_n)))\subset D_R(c_n)
    \]
    since $g'(0) = \frac{R\beta}{4}\left(\Phi_\M^{-1}\right)'(\Phi_\M(c_n))= \frac{R\beta}{4}\cdot\frac{1}{\beta}$.
    
    Note that $\Phi_\M(D_{R/8}(c_n))\subset D_{R\beta/4}(\Phi_\M(c_n))$ and that for parameters $c'$ and $c''$ we have $N(c') \leq N(c'') \Leftrightarrow |\Phi_\M(c')| \geq |\Phi_\M(c'')|$ where $N(c) := \min\{n\in\N\mid |p_c^{\circ n}(0)| > 2\}$ denotes the escape time\footnote{The escape radius can be any real number greater than $2$.}.
    Then it follows that the escape time of parameters $c\in D_{R/8}(c_n)$ is bounded from above and below by parameters $\{\Phi_\M(c_a), \Phi_\M(c_b)\} = \partial D_{R\beta/4}(\Phi_\M(c_n))\cap (1,\infty)\cdot e^{2\pi i \vartheta}$ where $\vartheta$ denotes the angle of the parameter ray $\mathcal{R}_\theta$ on which $c_n$ lies.
    The escape times of $c_a$ and $c_b$ in turn are bounded from above and below by the parameters $c'_a$ and $c'_b$ with the maximal and minimal escape times of parameters in $\partial D_R(c_n)\cap\mathcal{R}_\vartheta$. Since the radius $R$ decreases at least quadratically with respect to the distance $\alpha := |c_n - c_0|$  as $c_n \to c_0$, we can apply \Cref{the:main theorem}
    \begin{align*}
        N(\alpha + O(\alpha^2)) &= \pi \cdot \frac{2qn}{|\mu'_{qn}(c_0)|(\alpha + O(\alpha^2))} + O(1) = \pi\cdot\frac{2qn}{|\mu'_{qn}(c_0)|\alpha(1 + O(\alpha))} + O(1) \\
        &= \pi\cdot\frac{2qn}{|\mu'_{qn}(c_0)|\alpha}\cdot(1 + O(\alpha)) + O(1) = \pi\cdot\frac{2qn}{|\mu'_{qn}(c_0)|\alpha} + O(1)
    \end{align*}
    and thus for parameters $c\in D_{R/8}(c_n)$ their escape times $N(c)$ differ at most by a constant from $N(c_n)$.
\end{proof}

We continue with some examples of \Cref{the:main theorem} applied to some bifurcations. To do this for a given bifurcation $c_0$, we equally distribute many parameters on the boundary $D_{|\alpha|}(c_0)$. For each of these, we determine their escape time $N(\alpha)$ and then list the smallest $N(\alpha)$ for a given $|\alpha|$ in \Cref{tab:experiments}. Using \Cref{eq:main theorem}, we can see how $N(\alpha)\cdot\frac{|\alpha\mu'_{qn}|}{2qn}\to\pi$. The derivative of the multiplier map can be easily determined numerically.

\begin{table}
\begin{tabular}{llllll}
$c_0$                       & $qn$ & $\alpha$   & $N(\alpha)$ & $N(\alpha)\cdot\frac{|\alpha\mu'_{qn}(c_0)|}{2qn}$ & $\pi\cdot\frac{2qn}{|\alpha\mu'_{qn}(c_0)|} - N(\alpha)$ \\
\toprule 
$0.25+0.5i$                 & $4$  & $0.1$      & $10$        & $\textcolor{black}{2.828426915}$                   & 1.10720817                                               \\
                            &    & $0.01$     & $109$       & $\textcolor{blue}{3.}\textcolor{black}{082985337}$ & 2.072081699                                              \\
                            &    & $0.001$    & $1109$      & $\textcolor{blue}{3.1}\textcolor{black}{36725449}$ & 1.720816987                                              \\
                            &    & $0.0001$   & $11106$     & $\textcolor{blue}{3.141}\textcolor{black}{250932}$ & 1.208169866                                              \\
                            &    & $0.00001$  & $111070$    & $\textcolor{blue}{3.1415}\textcolor{black}{33774}$ & 2.081698659                                              \\
                            &    & $0.000001$ & $1110721$   & $\textcolor{blue}{3.14159}\textcolor{black}{3171}$ & -0.1830134112                                            \\ \toprule
$-1.125+0.2165063509i$      & $6$  & $0.1$      & $11$        & $\textcolor{blue}{3.}\textcolor{black}{299987013}$ & -0.5279832759                                            \\
                            &    & $0.01$     & $104$       & $\textcolor{blue}{3.1}\textcolor{black}{19987721}$ & 0.7201672408                                             \\
                            &    & $0.001$    & $1048$      & $\textcolor{blue}{3.14}\textcolor{black}{3987627}$ & -0.7983275915                                            \\
                            &    & $0.0001$   & $10470$     & $\textcolor{blue}{3.14}\textcolor{black}{0987639}$ & 2.016724085                                              \\
                            &    & $0.00001$  & $104718$    & $\textcolor{blue}{3.1415}\textcolor{black}{27637}$ & 2.167240848                                              \\
                            &    & $0.000001$ & $1047199$   & $\textcolor{blue}{3.1415}\textcolor{black}{84636}$ & 2.672408482                                              \\ \toprule
$-1.768529152467788+0i$     & 12   & $0.1$      & $4$         & $\textcolor{black}{7.606871243}$                   & -2.348023752                                             \\
                            &    & $0.01$     & $19$        & $\textcolor{blue}{3.}\textcolor{black}{61326384}$  & -2.48023752                                              \\
                            &     & $0.001$    & $166$       & $\textcolor{blue}{3.1}\textcolor{black}{56851566}$ & -0.8023751996                                            \\
                            &     & $0.0001$   & $1651$      & $\textcolor{blue}{3.1}\textcolor{black}{39736105}$ & 0.9762480041                                             \\
                            &     & $0.00001$  & $16518$     & $\textcolor{blue}{3.141}\textcolor{black}{25748}$  & 1.762480041                                              \\
                            &     & $0.000001$ & $165179$    & $\textcolor{blue}{3.141}\textcolor{black}{238462}$ & 18.62480041                                              \\ \toprule
$-0.125+0.649519052838329i$ & 3    & $0.1$      & $17$        & $\textcolor{black}{2.944493608}$                   & 1.137949073                                              \\
                            &     & $0.01$     & $180$       & $\textcolor{blue}{3.1}\textcolor{black}{17699115}$ & 1.379490733                                              \\
                            &     & $0.001$    & $1812$      & $\textcolor{blue}{3.1}\textcolor{black}{38483775}$ & 1.794907326                                              \\
                            &     & $0.0001$   & $18137$     & $\textcolor{blue}{3.141}\textcolor{black}{428269}$ & 0.9490732569                                             \\
                            &     & $0.00001$  & $181379$    & $\textcolor{blue}{3.1415}\textcolor{black}{84154}$ & 0.4907325693                                             \\
                            &     & $0.000001$ & $1813798$   & $\textcolor{blue}{3.14159}\textcolor{black}{801}$  & -3.092674307                                             \\ \bottomrule
\end{tabular}
\caption{We show the convergence to $\pi$ as per \Cref{eq:main theorem} for 4 bifurcations.}
\label{tab:experiments}
\end{table}

These experiments conclude our results.

\end{document}